\documentclass[a4paper,12pt]{amsart}

\usepackage[utf8]{inputenc}
\usepackage[english]{babel}
\usepackage{amscd,amssymb,amsthm,amsmath,csquotes}
\usepackage{marginnote} 
\usepackage{mathrsfs}
\usepackage{quiver}
\usepackage{float}
\usepackage[margin=1.5in]{geometry}
\usepackage{comment}
\usepackage{mathtools}
\usepackage{upgreek}
\usepackage{faktor}
\usepackage{changepage}
\usepackage{rotating}
\usepackage{enumerate}

\usepackage[backend=biber, style=alphabetic, sorting=nyt]{biblatex}
\newtheorem{theorem}{Theorem}[section]
\newtheorem{lemma}[theorem]{Lemma}

\newtheorem{proposition}[theorem]{Proposition}

\theoremstyle{definition}
\newtheorem{definition}[theorem]{Definition}

\theoremstyle{remark}
\newtheorem*{remark}{Remark}

\makeatletter
\newcommand{\colim@}[2]{%
  \vtop{\m@th\ialign{##\cr
    \hfil$#1\operator@font colim$\hfil\cr
    \noalign{\nointerlineskip\kern1.5\ex@}#2\cr
    \noalign{\nointerlineskip\kern-\ex@}\cr}}%
}
\newcommand{\colim}{%
  \mathop{\mathpalette\colim@{\rightarrowfill@\textstyle}}\nmlimits@
}
\makeatother

\makeatletter
\newcommand{\hocolim@}[2]{%
  \vtop{\m@th\ialign{##\cr
    \hfil$#1\operator@font hocolim$\hfil\cr
    \noalign{\nointerlineskip\kern1.5\ex@}#2\cr
    \noalign{\nointerlineskip\kern-\ex@}\cr}}%
}
\newcommand{\hocolim}{%
  \mathop{\mathpalette\hocolim@{\rightarrowfill@\textstyle}}\nmlimits@
}
\makeatother

\usepackage{hyperref}
\hypersetup{
    colorlinks,
    citecolor=black,
    filecolor=black,
    linkcolor=black,
    urlcolor=black
}

\newcommand{\Hom}{\operatorname{Hom}}

\newcommand{\A}{\mathcal{A}}
\newcommand{\B}{\mathcal{B}}
\newcommand{\C}{\mathcal{C}}

\newcommand{\T}{\mathcal{T}}

\newcommand{\op}[1]{{#1}^{\operatorname{op}}}

\usepackage[all,cmtip]{xy}

\newcommand{\Mod}{\operatorname{-Mod}}

\newcommand{\hot}{\operatorname{Hot}}

\newcommand{\Ker}{\operatorname{Ker}}

\newcommand{\HH}{\mathsf{HH}}
\newcommand{\tow}[1]{\overset{#1}{\to}}
\newcommand{\Hm}[1]{\Hom_{#1}}
\newcommand{\ke}{k[\varepsilon]}
\newcommand{\Barv}{\operatorname{Bar}}
\newcommand{\kn}{k[t]/(t^{n+1})}
\newcommand{\cdef}{\operatorname{cDef}}
\newcommand{\Gr}{\operatorname{Gr}}
\title[Hochschild cohomology and curved Morita deformations]{Hochschild cohomology parametrizes curved Morita deformations}
\author{Alessandro Lehmann}
\address[Alessandro Lehmann]{Universiteit Antwerpen, Departement Wiskunde, Middelheimcampus,
Middelheimlaan 1,
2020 Antwerp, Belgium}
\email{alessandro.lehmann@uantwerpen.be}
\address{SISSA, Via Bonomea 265, 34136 Trieste TS, Italy}
\email{alehmann@sissa.it}
\thanks{
This project has received funding from the European Research Council (ERC) under the European Union’s Horizon 2020 research and innovation programme (grant agreement No. 817762).
}

\makeatletter
\@namedef{subjclassname@2020}{
  \textup{2020} Mathematics Subject Classification}
\makeatother

\makeindex 

\begin{document}
\maketitle
\maketitle
\section{Introduction}
Let $k$ be a field, and $A$ be a dg algebra over $k$. In \cite{MoritaDef}, Keller and Lowen defined a canonical map $\nu$ from the set of equivalence classes of \emph{Morita deformations} of $A$ -- that is, dg $\ke$-algebras $B_\varepsilon$ that come equipped with a Morita bimodule identifying the derived categories of $A$ and $B_\varepsilon\otimes_{\ke}^{\operatorname{L}} k$ -- to the second Hochschild cohomology of $A$. Surprisingly they observed that this map is, in general, neither injective nor surjective. The lack of surjectivity in particular is related to the existence of deformations that are intrinsically curved, i.e. whose curvature cannot be removed by passing to a Morita equivalent base algebra. Similar results were later obtained by Lurie in \cite{DAGX}; since curved objects are significantly harder to study with homological methods than their uncurved counterparts, this is usually called the \emph{curvature problem}. Positive results, showing that in some cases it is indeed possible to substitute a curved deformation with an uncurved one were obtained in \cite{LowenCurvature, Blanc_Katzarkov_Pandit_2018} for formal deformations of certain categories and in \cite{LOWEN2013441} for infinitesimal deformations of schemes. In this paper we show that, if one allows for curved deformations and considers them up to an appropriate notion of equivalence, the same map $\nu$ introduced in \cite{MoritaDef} becomes a bijection. Concretely, for us a curved Morita deformation is a $\ke$-free cdg $\ke$-algebra $B_\varepsilon$ whose reduction $B=B_\varepsilon\otimes_{\ke}k$ is a dg $k$-algebra equipped with a $B$-$A$ Morita bimodule; an equivalence of deformations $B_\varepsilon$ and $C_\varepsilon$ is an appropriately cofibrant cdg $B_\varepsilon$-$C_\varepsilon$ bimodule whose reduction is a Morita equivalence compatible with the equivalences to $A$. Note how, since we are dealing with (bi)modules over curved algebras, it does not make sense to ask for cofibrancy in the usual sense and a different condition has to be considered; this variation in the notion of equivalence of deformations is what leads to the unconditional injectivity result. This indeed solves the curvature problem for first order deformations: we interpret any Hochschild class as a curved deformation, and those have a meaningful homological interpretation via their $1$-derived category (introduced in \cite{Nder}). We also show that a cdg bimodule gives an equivalence of curved deformations precisely when it induces an equivalence between the respective $1$-derived categories.

The results of this paper fit into the larger picture explained in the introduction of \cite{Nder}: the bijection introduced here gives the upper arrow in the commutative diagram of isomorphisms \begin{equation*}\label{eqsquare}
    \xymatrix{{\HH^2(A)} \ar[d]_{\alpha} \ar[r]^-{\nu^{-1}} & {{\cdef_A(\ke)}} \ar[d]^{\beta} \\
    {\HH^2(D_{\mathrm{dg}}(A))} \ar[r]_-{\eta} & {\mathrm{Def}^{\mathrm{cat}}_{D_{\mathrm{dg}}(A)}(k[\epsilon])}
    }
\end{equation*}
where $\alpha$ is the characteristic morphism for Hochschild cohomology from \cite{lowchar} and $\beta$ assigns to a curved deformation its $1$-derived category. The arrow $\eta$ will be introduced in \cite{FutDef} as a bijection between the Hochschild cohomology of any pretriangulated dg category $\T$ and a set of appropriately defined categorical square zero extensions of $\T$.
The results in the present paper, other than giving the bijection $\nu$, show that the arrow $\beta$ in the diagram is well-defined; this highlights how the $1$-derived category, besides admitting an interpretation as a deformation of the derived category of the base, also detects equivalences of curved deformations. These facts, taken together, yield a particularly nice picture: given any Hochschild class we interpret it as a curved deformation, to which we associate its $1$-derived category; moreover this is a deformation of the derived category of the base, which corresponds appropriately to the starting class; the last statement is currently work in progress and will be fully developed in \cite{FutDef}.

Let us highlight that the novelty of our approach here consists in giving an explicit description in terms of algebras and bimodules of the set of curved deformations: indeed, we expect that our constructions could be generalized to describe the full groupoid of curved deformations over any local artinian (dg) algebra, and the resulting deformation functor (formal moduli problem, in the language of \cite{DAGX}) to be controlled by the full Hochschild complex. 
\vspace{8pt}

\emph{Acknowledgements.}
The author is thankful to Wendy Lowen for the continued support and the many suggestions in the redaction of this paper, and to Bernhard Keller for a helpful conversation. He also thanks the anonymous referee for highlighting some mistakes and imprecisions that were present in the original version.

\section{Curved Morita deformations}

\subsection{Cdg algebras and modules} We now give the basic definitions of cdg algebras and modules; for a more detailed introduction, see \cite{Nder} and the references therein. Let $R$ be a commutative ring. A cdg algebra $\A$ over $R$ is given by a graded $R$-algebra $\A^\#$ equipped with a derivation $d_\A\in \Hom_R(\A^\#,\A^\#)$ of degree $1$ and an element $c\in A^2$ such that $d_\A(c)=0$ and $d^2_\A=[c, -]$; the derivation $d$ is called predifferential and the element $c$ curvature of the algebra. A right cdg module over a cdg algebra $\A$ is a graded right $\A^\#$-module $M^\#$ equipped with an $\A^\#$-derivation $d_M\in \Hom_R(M^\#,M^\#)^1$ such that $d^2_M m= -mc$ for all $m\in M$; if $\A$ and $\B$ are cdg algebras, a cdg $\A$-$\B$ bimodule is a graded $\A$-$\B$ bimodule $M$ equipped with a degree $1$ derivation $d_M$ for which the identity $d_M^2m=c_\A m-m c_\B$ holds. For us, a morphism of cdg algebras is a morphism of graded algebras which preserves the curvature and commutes with the multiplications and predifferentials. Any morphism of cdg algebras $\A\to \B$ endows $\B$ with the structure of a cdg $\A$-$\B$ bimodule. Note that, unlike in the dg case, a cdg $\A$-$\B$ bimodule does not have an underlying structure of a cdg $\A$ or $\B$ module; in particular the cdg $\A$-bimodule $\A$ is not, in general, a cdg $\A$-module.
\subsection{Hochschild cohomology and Morita bimodules}
Let $A, B$ be dg algebras over $k$ and $X$ an $A$-$B$ bimodule. Recall that $X$ is said to be a (derived) Morita bimodule if the functor \[
-\otimes_{A}^{\operatorname{L}}X\colon D(A)\to D(B)
\] is an equivalence.
Following \cite{MoritaDef}, we will use the \emph{arrow category} $\mathfrak{c}_X$ associated to the bimodule $X$, which is a dg-category with two distinct objects $P$ and $Q$ and \[
\Hm{\mathfrak{c}_X}(P,P)=A, \, \Hm{\mathfrak{c}_X}(Q,Q)=B, \, \Hm{\mathfrak{c}_X}(Q,P)=X \text{ and }\Hm{\mathfrak{c}_X}(P,Q)=0.
\] with composition defined in the obvious way. The category $\mathfrak{c}_X$ comes equipped with two fully faithful dg functors \[
j\colon A \to \mathfrak{c}_X \text{ and }i \colon B \to \mathfrak{c}_X
\] where $A$ and $B$ are seen as dg-categories with one object, which are identified respectively with $P$ and $Q$.
In \cite{KellerInvariance}, Keller introduced for any Morita bimodule $X$ a bijection\footnote{In fact $\varphi_X$ is shown to give an isomorphism in the homotopy category of $B_\infty$ algebras between the Hochschild complexes, but we'll only need the result at the homotopy level.} \[
\varphi_X\colon \HH^\bullet(B)\to \HH^\bullet(A)
\] between the Hochschild cohomologies of $B$ and $A$ which is functorial with respect to the (derived) tensor product of bimodules. By construction of $\varphi_X$, the diagram
\begin{equation*}\label{phiX}
\begin{tikzcd}
	& {\HH^\bullet(B)} \\
	{\HH^\bullet(\mathfrak{c}_X)} \\
	& {\HH^\bullet(A)}
	\arrow["{\varphi_X}", from=1-2, to=3-2]
	\arrow["{i^*}", from=2-1, to=1-2]
	\arrow["{j^*}"', from=2-1, to=3-2]
\end{tikzcd}    
\end{equation*}
is a commutative diagram of isomorphisms, where $j^*, i^*$ are the maps induced by the restrictions along the fully faithful functors $A\to \mathfrak{c}_{X}$ and $B\to \mathfrak{c}_{X}$.
\subsection{Curved Morita deformations}
Denote with $\ke$ the algebra of the dual numbers. Let $A$ be a dg algebra over $k$. In the following, we will just say $\ke$-free to mean free as a graded $\ke$-module.
\begin{definition}
A curved Morita deformation of $A$ is  a $\ke$-free cdg $\ke$-algebra $B_\varepsilon$ equipped with, setting $B:=B_\varepsilon\otimes_{\ke}k$, a $B$-$A$ Morita bimodule $X$. Two curved deformations $(B_\varepsilon, X)$ and $(C_\varepsilon, Y)$ are equivalent if there exists a  cdg $B_\varepsilon$-$C_\varepsilon$ bimodule $Z_\varepsilon$ that is free as a graded $B_\varepsilon$-module and as a graded $C_\varepsilon$-module (in particular, it is $\ke$-free) and such that, setting $C=C_\varepsilon\otimes_{\ke} k$, the $B$-$C$ bimodule $Z=Z_\varepsilon \otimes_{\ke} k$ is cofibrant as a bimodule and there exists an isomorphism $X\cong Z\otimes_C Y$ in the derived category of $B$-$A$ bimodules.
\end{definition}
 We will denote with $\cdef_A(\ke)$ the set of curved Morita deformations of $A$ up to equivalence. 
\begin{lemma}
    Equivalence of curved deformations is an equivalence relation, so the set $\cdef_A(\ke)$ is well defined.
\end{lemma}
\begin{proof}
     Transitivity is easy, since if a deformation $B_\varepsilon$ is equivalent via $X_\varepsilon$ to $C_\varepsilon$ and $C_\varepsilon$ is equivalent to $D_\varepsilon$ via $Y_\varepsilon$ then one can check that $X_\varepsilon\otimes_{C_\varepsilon} Y_\varepsilon$ gives an equivalence between $B_\varepsilon$ and $D_\varepsilon$.
    Symmetry is less immediate: let $Z_\varepsilon$ be a morphism between deformations $B_\varepsilon$ and $C_\varepsilon$. Then $Z$ is a Morita bimodule, thus there exists a cofibrant $C$-$B$ Morita bimodule $W$ such that $Z\otimes_C W\cong B$ and $W\otimes_B Z\cong C$. Denoting with $\mathfrak{c}_W$ the arrow category of $W$, by definition of $\varphi_W$ there is a commutative diagram of isomorphisms 
\[\begin{tikzcd}
	& {\HH^2(C)} \\
	{\HH^2(\mathfrak{c}_W)} \\
	& {\HH^2(B).}
	\arrow["{\varphi_W}", from=1-2, to=3-2]
	\arrow["{i^*}", from=2-1, to=1-2]
	\arrow["{j^*}"', from=2-1, to=3-2]
\end{tikzcd}\]
Since $\varphi_W$ is an inverse to $\varphi_Z$ and by Proposition \ref{welldef} the map $\varphi_Z$ carries the class that defines $B_\varepsilon$ to the one defining $C_\varepsilon$, reasoning exactly as in the proof of Proposition \ref{injective} we can conclude that there exists a graded $B_\varepsilon$-free and $C_\varepsilon$-free $C_\varepsilon$-$B_\varepsilon$ cdg bimodule $\hat{W_\varepsilon}$ such that its reduction $\hat{W}$ is cofibrant and isomorphic in the derived category of bimodules to $W$, which is a Morita equivalence. We have thus proved that equivalence of curved deformations is a symmetric relation. Reflexivity is proven in the same way: for any deformation $B_\varepsilon$ we have the cdg $B_\varepsilon$-bimodule $B_\varepsilon$ whose reduction $B$ is manifestly a Morita bimodule; resolving this via the same procedure as before, we are done.
\end{proof}
\begin{remark}
It is clear that the set $\cdef_A(\ke)$ is really just the set of connected components of a (higher) groupoid. A satisfactory construction of the whole groupoid would require a generalization to curved categories of the Morita theory of \cite{toen_derived_morita} where the weak equivalences are, at least on $\ke$-free categories, the morphisms inducing weak equivalence on the (uncurved) reductions. In this generalization, the mapping spaces would be given by (an enhancement of) some version of the semiderived category of bimodules, following constructions in \cite{Positselski_2018}; indeed our condition of graded $C_\varepsilon$-freeness and cofibrancy of the reduction is implied by the natural cofibrancy condition for the semiderived category of \cite{Positselski_2018}. This is also related to the $1$-derived category of \cite{Nder}, since its subcategory given by the $\ke$-free modules coincides with the semiderived category. At the moment, this theory does not exist so we have to define our objects ``by hand''. In particular, we have no intrinsic definition of the left derived functor of the reduction $-\otimes_{\ke}k$ from (weakly) curved dg $\ke$-algebras to dg $k$-algebras; we solve this issue by restricting to $\ke$-free algebras to begin with.
\end{remark}

\subsection{Curved deformations and Hochschild cohomology}
It is well-known (see \cite{lowchar, MoritaDef}) that there is a correspondence between the set of Hochschild 2-cocycles and that of deformations of $A$ as a $cA_\infty$ algebra, i.e. $\ke$-free $cA_\infty$ algebras $A_\varepsilon$ equipped with an isomorphism of dg algebras $A_\varepsilon\otimes_{\ke}k\cong A$. Define a map \[\nu \colon \cdef_A(\ke)\to \HH^2(A)\] in the same way as in \cite{MoritaDef}: if $(B_\varepsilon, X)$ is a curved deformation of $A$, then $B_\varepsilon$ is a $cA_\infty$ deformation of $B$, so defines a cocycle $\eta$ and a class $[\eta]$ in $\HH^2(B)$. The class $\nu(B_\varepsilon)\in \HH^2(A)$ is given by definition by $\varphi_X([\eta])$.
\begin{proposition}\label{welldef}
The map $\nu$ is well defined.
\end{proposition}
\begin{proof}
The proof of \cite[Proposition 3.3]{MoritaDef} applies verbatim.
\end{proof}
\begin{proposition}
    The map $\nu$ is surjective.
\end{proposition}
\begin{proof}
Let $[\eta]\in \HH^2(A)$ be an Hochschild class; then any cocycle $\eta$ representing the class defines a $cA_\infty$ deformation $A_\varepsilon$ of $A$. Assume for now that there exists a $\ke$-free cdg algebra $B_\varepsilon$ equipped with a $cA_\infty$ $B_\varepsilon$-$A_\varepsilon$ bimodule $Z_\varepsilon$ such that the reduction $Z$ is a Morita $B$-$A$ bimodule (in particular, cannot have higher components). It is clear that $(B_\varepsilon, Z)$ is an element of $\cdef_A(\ke)$. By definition of $\varphi_Z$, the diagram
\[\begin{tikzcd}
	& {\HH^2(B)} \\
	{\HH^2(\mathfrak{c}_Z)} \\
	& {\HH^2(A)}
	\arrow["{\varphi_Z}", from=1-2, to=3-2]
	\arrow["{i^*}", from=2-1, to=1-2]
	\arrow["{j^*}"', from=2-1, to=3-2]
\end{tikzcd}\]
commutes, and $\mathfrak{c}_{Z_\varepsilon}$ defines a $cA_\infty$ deformation of $\mathfrak{c}_Z$ which defines an Hochschild cocycle $ \mu\in \HH^2(\mathfrak{c}_Z)$. Moreover by construction, $j^*[\mu]=[\eta]$; therefore, $\nu(B_\varepsilon)=j^* [\mu]=[\eta]$. We are left to prove that such $B_\varepsilon$ exists. Consider the cdg algebra $\mathcal{Y}(A_\varepsilon)$ given as the image via the curved Yoneda embedding (see \cite{FilteredAinf} for a complete description of the curved Yoneda embedding and of the category of $qA_\infty$ modules) of the $cA_\infty$ algebra $A_\varepsilon$. This is isomorphic to the cdg algebra $\Hm{A_\varepsilon\Mod}(A_\varepsilon, A_\varepsilon)$ where the hom is taken in the cdg category of $qA_\infty$ $A_\varepsilon$-modules. Since $A_\varepsilon$ is $\ke$-free, the same holds for $\mathcal{Y}(A_\varepsilon)$ since, as a graded $\ke$-module, it is a product of homs between $\ke$-free modules. There is a natural map of $cA_\infty$ algebras $A_\varepsilon \tow{\mathcal{Y}} \mathcal{Y}(A_\varepsilon)$ given by the curved Yoneda embedding. Its higher components are killed by reduction, so setting $\mathcal{Y}(A_\varepsilon)_0:=\mathcal{Y}(A_\varepsilon)\otimes_{\ke}k$, we have a morphism of dg algebras $A\tow{\mathcal{Y}_0} \mathcal{Y}(A_\varepsilon)_0$. This coincides with the $A_\infty$ Yoneda embedding which, by \cite[Theorem 4.15]{FilteredAinf} is a quasi-isomorphism and we are done.
\end{proof}

\begin{proposition}\label{injective}
    The map $\nu$ is injective.
\end{proposition}
\begin{proof}
    Suppose that $(B_\varepsilon, X)$ and $(C_\varepsilon, Y)$ are cdg deformations of $A$ such that $\nu(B_\varepsilon)=\nu(C_\varepsilon)$. Let $Z$ be a cofibrant Morita $B$-$C$ bimodule such that there is an isomorphism $Z\otimes_C Y\cong X$ in the derived category -- such a bimodule always exists by the standard Morita theory of dg-algebras, see e.g. the proof of \cite[Proposition 3.7]{MoritaDef}. Then by definition of $\varphi_Z$ the diagram
\[\begin{tikzcd}
	& {\HH^2(B)} \\
	{\HH^2(\mathfrak{c}_Z)} && {\HH^2(A)} \\
	& {\HH^2(C)}
	\arrow["{j^*}"', from=2-1, to=3-2]
	\arrow["{i^*}", from=2-1, to=1-2]
	\arrow["{\varphi_X}", from=1-2, to=2-3]
	\arrow["{\varphi_Y}"', from=3-2, to=2-3]
\end{tikzcd}\]
is commutative. Denoting with $\eta_B$ and $\eta_C$ the Hochschild cocycles of $B$ and $C$ defining their $cA_\infty$ deformations $B_\varepsilon$ and $C_\varepsilon$, since they map to the same element in $\HH^2(A)$ and all arrows are isomorphisms, there must be an element $[\gamma]\in \HH^2(\mathfrak{c}_Z)$ such that \[i^*([\gamma])=[\eta_\B] \text{ and }j^*([\gamma])=[\eta_\C].\] By \cite[Lemma 3.8]{MoritaDef}, it is actually possible to find a 2-cocycle $\gamma$ mapping to $\eta_\A$ and $\eta_\B$ before passing to cohomology. Now $\gamma$ defines a $cA_\infty$ deformation of $\mathfrak{c}_Z$ which is immediately seen to be itself an arrow category for some $B_\varepsilon$-$C_\varepsilon$ $cA_\infty$ bimodule $Z_\varepsilon$. By construction of $Z_\varepsilon$ it is $\ke$-free and its reduction modulo $\varepsilon$ is the $B$-$C$ Morita bimodule $Z$.

We are almost done, except that we need to rectify $Z_\varepsilon$ to a cdg -- and not $c{A}_\infty$ -- bimodule. This is an application of Koszul duality, for which we employ the notations and constructions of \cite[Sections 6 and 8]{Positselski_2011}\footnote{In principle the constructions there are only given for a base field, but since in our case everything is (graded) $\ke$-free we can repeat verbatim his constructions over this base ring. In particular, all the tensor products in the various bar constructions are intended over $\ke$.}. Using the fact that the (co)bar construction is appropriately monoidal, a $cA_\infty$ $B_\varepsilon$-$C_\varepsilon$ bimodule corresponds to a $cA_\infty$ module over the cdg algebra $E_\varepsilon=\op{B_\varepsilon}\otimes_{\ke} C_\varepsilon$. This by definition is a cdg comodule $\Barv_v(E_\varepsilon, Z_\varepsilon)$ over the cdg coalgebra $\Barv_v(E_\varepsilon)$ (for a similar construction, see the curved bar construction of \cite{Positselski_2018}). Consider then the cdg $E_\varepsilon$-module $\hat{Z_\varepsilon}=E_\varepsilon \otimes^\tau \Barv_v(E_\varepsilon, Z_\varepsilon)$. This is graded $E_\varepsilon$-free -- its underlying graded module is $E_\varepsilon\otimes_{\ke}\Barv_v(E_\varepsilon, Z_\varepsilon)$ -- so it is both graded $C_\varepsilon$-free and graded $B_\varepsilon$-free. Moreover, denoting $E=E_\varepsilon \otimes_{\ke} k \cong \op{B}\otimes_k C$, the reduction $\hat{Z}$ of $\hat{Z_\varepsilon}$ is the dg $E$-module  $E\otimes ^\tau \Barv_v(E, Z)$, where now the tensors are over the base field $k$. By the proof of \cite[Theorem 6.3]{Positselski_2011}, this coincides with the (reduced) bar resolution of the dg $E$-module $M$, which is therefore both cofibrant and, being isomorphic to $Z$ in the derived category of bimodules, a Morita equivalence. 
\end{proof}

\section{Filtered Morita equivalences}
In this section we'll consider deformations not necessarily of the first order, putting ourselves in the generality of \cite{Nder}. For any cdg algebra $\A$, we will denote with $\A\Mod$ the dg category of right cdg $\A$-modules.
\subsection{The $n$-derived category}We begin by recalling some notation and constructions from \cite{Nder}. Let $A$ be a dg algebra, and $A_n$ a cdg deformation of $A_n$ over $R_n=\kn$ (in the sense of \cite{Nder}, i.e. a cdg algebra structure on $A\otimes_k R_n$ which reduces to the algebra structure of $A$). The $n$-derived category $D^n(A_n)$ of $A_n$ is defined as the quotient of the homotopy category $\hot(A_n)=H^0(A_n\Mod)$ of cdg $A_n$-modules by the subcategory given by the modules $M$ for which the associated graded $\Gr_t (M)$ with respect to the $t$-adic filtration is acyclic. The quotient functor $\hot(A_n)\to D^n(A_n)$ has a left adjoint, denoted $\mathbf{p}_n$ and the modules in the essential image of this functor are deemed $n$-homotopy projective. In the uncurved setting, we will denote with $\mathbf{p}$ the left adjoint to the quotient functor $\hot(A)\to D(A)$.
 Let $A$ and $B$ be two dg algebras and $A_n$ and $B_n$ deformations over $R_n$ of $A$ and $B$ respectively. Let $X_n$ be a cdg $A_n$-$B_n$ bimodule. This induces an adjoint pair of dg-functors  \[\begin{tikzcd}
	{A_n\Mod} & {B_n\Mod}
	\arrow["{- \otimes_{A_n} X_n}"', shift right, from=1-1, to=1-2]
	\arrow["{\Hm{B_n}(X_n,-)}"', shift right, from=1-2, to=1-1].
\end{tikzcd}\]
In the following we will assume that $X_n$ is projective as a graded $B_\varepsilon$-module and that the $A$-$B$ bimodule $X=X_n\otimes_{R_n}k$ is cofibrant as a $B$-module. Note that this condition is implied by cofibrancy as a bimodule.
\begin{lemma}\label{homcommutes}
    Let $X_n$ be an $A_n$-$B_n$ bimodule that is projective as a graded $B_n$-module and such that $X=X_n\otimes_{R_n}k$ is cofibrant as a dg $B$-module. Then the functor \[
    \Hm{B_n}(X_n,-)\colon B_n\Mod \to A_n\Mod
    \] preserves $n$-acyclic modules.
\end{lemma}
\begin{proof}
    Let $M$ be a cdg $B_n$-module. Recall that by \cite[Proposition 3.2]{Nder} $M$ is $n$-acyclic if and only if $t^i\Ker t^{i+1}_M$ is acyclic for $i=0, \ldots, n$. We first show that \begin{equation}\label{tiso}
    t^i\Hm{B_n}(X_n, M)\cong \Hm{B_n}(X_n, t^iM).    \end{equation}
There is a natural map \[
    t^i\Hm{B_n}(X_n, M)\to \Hm{B_n}(X_n, t^iM) 
    \] induced by the surjection $M\tow{t^i}t^iM$. This is immediately seen to be injective and, since $X_n$ is projective as a graded $B_n$-module, it is also surjective. Applying now the exact functor $\Hm{B_n}(X_n,-)$ to the short exact sequence \[
    0\to t^{i+1}M\to t^iM \to \Gr^i_t(M)\to 0
    \] and using the isomorphism (\ref{tiso}), we find that there is an isomorphism \[
  \Hm{B}(X, \Gr^i_t(M))\cong \Gr^i_t(\Hm{B_n}(X_n, M)).
    \]

    If $M$ is $n$-acyclic, then $\Gr_t^i(M)$ is acyclic; so since $X$ is cofibrant as a $B$-module, the complex $\Hm{B}(X, \Gr_t^i(M))$ has to be acyclic.
\end{proof}

\begin{remark}
    (The proof of) this Lemma has some non-obvious consequences, most notably the fact that than any $A_n$-module that is projective as a graded $A_n$-module whose reduction is homotopy projective is $n$-homotopy projective. In particular, the $\ke$-module \[
    \ldots \to \ke \tow{\varepsilon}\ke\tow{\varepsilon}\ke \to \ldots 
    \]
    is $1$-homotopy projective.
\end{remark}
As a consequence of this Lemma, we can derive the functor $\Hm{B_n}(X_n,-)$ without the need to take an injective resolution in the second variable  -- morally, we have resolved $X_n$. In particular, any morphism of cdg algebras $A_n\to B_n$ gives to $B_n$ the structure of a cdg $A_n$-$B_n$ bimodule satisfying the hypotheses above. We have then a functor \[
\Hm{B_n}(X_n,-)  \colon D^n(B_n)\tow{}D^n(A_n).
\]
\begin{proposition}\label{derivedadjoints}
    The functor $-\otimes_{A_n}X_n$ admits a left derived functor $-\otimes_{A_n}^{\operatorname{L}}X_n$, giving a derived adjoint pair
    \[\begin{tikzcd}
	{D^n(A_n)} & {D^n(B_n).}
	\arrow["{-\otimes_{A_n}^{\operatorname{L}}X_n}"', shift right, from=1-1, to=1-2]
	\arrow["{\Hm{B_n}(X_n,-) }"', shift right, from=1-2, to=1-1]
\end{tikzcd}\]
\end{proposition}
\begin{proof}
    We define the left derived functor $-\otimes_{A_n}^{\operatorname{L}}X_n$ as the composition 
\[\begin{tikzcd}
	{D^n (A_n)} & {\hot (A_n)} & {\,\hot(B_n)} & {D^n (B_n).}
	\arrow["{\mathbf{p}_n}", from=1-1, to=1-2]
	\arrow["{-\otimes_{A_n}X_n}", from=1-2, to=1-3]
	\arrow[from=1-3, to=1-4]
\end{tikzcd}\]To show that this is a left adjoint to $\Hm{B_n}(X_n,-)$, one uses the fact that the functor $\Hm{B_n}(X_n,-)$ preserves acyclic objects and consequently its left adjoint $-\otimes_{A_n}X_n$ preserves $n$-homotopy projective modules.
\end{proof}
The bimodule $X_n$ induces $A_i$-$B_i$ bimodules $X_i$ for $i=0, \ldots, n-1$; we set $X_0=X$. The main result of this section is the following:
\begin{proposition}\label{preservation}
       The adjoint pair \[\begin{tikzcd}
	{D^n(A_n)} & {D^n(B_n)}
	\arrow["{-\otimes_{A_n}^{\operatorname{L}}X_n}"', shift right, from=1-1, to=1-2]
	\arrow["{\Hm{B_n}(X_n,-)}"', shift right, from=1-2, to=1-1]
\end{tikzcd}\]
is an equivalence if and only if the $A$-$B$ bimodule $X$ is a Morita bimodule.
\end{proposition}
This is the union of Propositions \ref{pres1} and \ref{pres2}, the proofs of which will take up the rest of the section. Denote with $F_i$ the restriction of scalars $A_i\Mod \to A_n\Mod$ along the projection $A_n\to A_i$.

\begin{lemma}\label{zeroaction}
Let $M$ be an $A_n$-module such that $t^{i+1}M=0$. Then $M$ has a natural structure of $A_i$-module and there is an isomorphism of $A_n$-modules
\[
M\otimes_{A_n}X_n\cong F_i(M\otimes_{A_i}X_i).
\]

\end{lemma}
\begin{proof}
    The first statement follows from the fact that the functor $F_i$ identifies the category $A_i\Mod$ as the full subcategory of $A_n\Mod$ given by the modules $M$ for which $t^{i+1}M=0$. To prove the second, observe that since the functor $F_i$ is fully faithful and the functor $A_i\otimes_{A_n}-$ is left adjoint to $F_i$ there is a natural isomorphism of $A_i$-modules \begin{equation}\label{forgetful}
        M\cong M\otimes_{A_n} A_i;
    \end{equation} we also know that $t^{i+1}M=0$ implies $t^{i+1}(M\otimes_{A_n}X_n)=0$; therefore $M\otimes_{A_n} X_n$ has the structure of a $B_i$-module and we can apply to it the isomorphism (\ref{forgetful}) to obtain a natural isomorphism of $B_i$-modules\[
    M\otimes_{A_n}X_n\cong  M\otimes_{A_n} X_n \otimes_{B_n}B_i.
    \] Since the diagram 
\[\begin{tikzcd}
	{A_n\Mod} & {A_i\Mod} \\
	{B_n\Mod} & {B_i\Mod}
	\arrow["{-\otimes_{A_n}A_i}", from=1-1, to=1-2]
	\arrow["{-\otimes_{A_n}X_n}"', from=1-1, to=2-1]
	\arrow["{-\otimes_{A_i}X_i}", from=1-2, to=2-2]
	\arrow["{-\otimes_{B_n}B_i}", from=2-1, to=2-2]
\end{tikzcd}\] commutes up to natural isomorphism, we have further isomorphisms \[
M\otimes_{A_n}X_n\cong B_i\otimes_{B_n}M\otimes_{A_n}X_n\cong M\otimes_{A_n} A_i \otimes_{A_i} X_i\cong M\otimes_{A_i}X_i
\]of $B_i$-modules and we are done.
\end{proof}
\begin{lemma}\label{compatibility}
    For any $A_n$-module $M$ there is a natural isomorphism
    \[
    \Gr_t(M\otimes_{A_n} X_n)\cong \Gr_t(M)\otimes_A X.
    \] 
    \end{lemma}
    \begin{proof}
    We prove that there are isomorphisms 
     \[
\Gr_t^i(M\otimes_{A_n}X_n)=\frac{t^i(M\otimes_{A_n}X_n)}{t^{i+1}(M\otimes_{A_n}X_n)}\cong  \frac{t^iM}{t^{i+1}M} \otimes_A X=\Gr_t^i(M)\otimes_A X
    \] for each $i$.
   By definition of the action of $t^i$ on $M\otimes_{A_n}X_n$, we have \[
    t^i(M\otimes_{A_n}X_n)\cong  t^iM\otimes_{A_n}X_n. 
    \]
    so by Lemma \ref{zeroaction} there is an isomorphism \[
    t^i(M\otimes_{A_n}X_n)\cong t^iM\otimes_{A_{n-i}}X_{n-i} . 
    \]Since the diagram 
\[\begin{tikzcd}
	{A_{n-i}\Mod} & A\Mod \\
	{B_{n-i}\Mod} & B\Mod
	\arrow["{-\otimes_{A_{n-i}}A}", from=1-1, to=1-2]
	\arrow["{-\otimes_{A_{n-i}}X_{n-i}}"', from=1-1, to=2-1]
	\arrow["{-\otimes_{A}X}", from=1-2, to=2-2]
	\arrow["{-\otimes_{B_{n-i}}B}", from=2-1, to=2-2]
\end{tikzcd}\] commutes up to natural isomorphism, we get isomorphisms \[\begin{split}
 \frac{t^i(M\otimes_{A_n}X_n)}{t^{i+1}(M\otimes_{A_n}X_n)}&\cong B \otimes_{B_{n-i}}t^i(M\otimes_{A_n}X_n) \cong
 t^iM \otimes_{A_{n-i}} X_{n-i} \otimes_{B_{n-i}}B
 \\
 &\cong  t^iM \otimes_{A_{n-i}} A \otimes_A X    \cong \frac{t^iM}{t^{i+1}M} \otimes_A X .
\end{split}\]
     \end{proof}
\begin{proposition}\label{pres1}
    If $X$ is a Morita bimodule, then
\[\begin{tikzcd}
	{D^n(A_n)} & {D^n(B_n).}
	\arrow["{-\otimes_{A_n}^{\operatorname{L}}X_n}"', shift right, from=1-1, to=1-2]
	\arrow["{\Hm{B_n}(X_n,-)}"', shift right, from=1-2, to=1-1]
\end{tikzcd}\]
is an equivalence.
\end{proposition}
\begin{proof}
    Recall that by definition, $ {M\otimes_{A_n}^{\operatorname{L}}}X_n={\mathbf{p}_nM\otimes_{A_n}^{\operatorname{}}} X_n$. We want to show that the unit and counit of the derived adjunction are isomorphisms; the unit is a morphism in $D^n(A_n)$ \[M\to \Hm{B_n}(X_n,\mathbf{p}_nM \otimes_{A_n}  X_n)\] and the counit a morphism in $D^n(B_n)$ \[ \mathbf{p}_n\Hm{B_n}(X_n,N) \otimes_{A_n} X_n  \to N;\]
    The unit is represented by the roof \[
    M \leftarrow \mathbf{p}_nM \to \Hm{B_n}(X_n,\mathbf{p}_nM\otimes_{A_n} X_n)
    \]where the left map is the canonical map $\mathbf{p}_nM \to M$ and the right is induced by the unit of the underived adjunction between $\Hm{B_n}(X_n,-)$ and ${-\otimes_{A_n}X_n}$;
    the counit is given by the composition \[
    \mathbf{p}_n \Hm{B_n}(X_n,N)\otimes_{A_n}X_n   \to   \Hm{B_n}(X_n,N) \otimes_{A_n}X_n \to N,
    \] where the first morphism is induced by $\mathbf{p}_n \Hm{B_n}(X_n,N)  \to \Hm{B_n}(X_n,N)$ and the second is the counit of the underived adjunction.
    To see that a morphism $ X \to Y$ is an isomorphism in the $n$-derived category is is sufficient to check that $\Gr_t(X)\to \Gr_t(Y)$ is a an isomorphism in $D(A)$; the morphism induced on the associated graded by the unit is the roof

\begin{equation}\label{roofs}
    \begin{split}
&\Gr_t(M) \leftarrow \Gr_t(\mathbf{p}_nM) \to \Gr_t(\Hm{B_n}(X_n,\mathbf{p}_nM\otimes_{A_n} X_n))\cong\\ &\cong  \Hm{B}(X, \Gr_t(\mathbf{p}_nM\otimes_{A_n} X_n))\overset{}{\cong} \Hm{B}(X, \Gr_t(\mathbf{p}_nM\otimes_A X)\end{split}\end{equation}
where the first isomorphism is due to (the proof of) Lemma \ref{homcommutes} while the second is Lemma \ref{compatibility}.
Now, since by \cite[Lemma 5.6]{Nder}
 $\Gr_t(\mathbf{p}_nM) \to \Gr_t(M)$ gives a homotopy projective resolution of the $A$-module $\Gr_t(M)$, this is nothing but the unit of the adjunction induced by $X$ between $D(A)$ and $D(B)$ applied to $\Gr_t(M)$: since that adjunction is an equivalence, (\ref{roofs}) is an isomorphism. The argument for the case of the counit is analogous.
\end{proof}
\begin{proposition}\label{pres2}
   If \[\begin{tikzcd}
	{D^n(A_n)} & {D^n(B_n).}
	\arrow["{-\otimes_{A_n}^{\operatorname{L}}X_n}"', shift right, from=1-1, to=1-2]
	\arrow["{\Hm{B_n}(X_n,-)}"', shift right, from=1-2, to=1-1]
\end{tikzcd}\] is an equivalence, then $X$ is a Morita bimodule.
\end{proposition}
\begin{proof}
    The diagrams
\[\begin{tikzcd}
	{D(A)} & {\hot(A)} & A\Mod & B\Mod \\
	{D^n(A_n)} & {\hot(B_n)} & {A_n\Mod} & {B_n\Mod}
	\arrow["{\mathbf{p}}", from=1-1, to=1-2]
	\arrow["{F_0}"', from=1-1, to=2-1]
	\arrow["{F_0}", from=1-2, to=2-2]
	\arrow["{-\otimes_A X}", from=1-3, to=1-4]
	\arrow["{F_0}"', from=1-3, to=2-3]
	\arrow["{F_0}"', from=1-4, to=2-4]
	\arrow["{\mathbf{p}_n}"', from=2-1, to=2-2]
	\arrow["{-\otimes_{A_n}X_n}", from=2-3, to=2-4]
\end{tikzcd}\]
commute, the first due to the fact that any homotopy projective $A$-module is $n$-homotopy projective as an $A_n$-module while the second is Lemma \ref{zeroaction}. Therefore the unit and counit of the adjunction between $D(A)$ and $D(B)$ induced by $X$ coincide with the unit and counit of the adjunction induced between $D^n(A_n)$ and $D^n(B_n)$ by $X_n$, applied to the modules in the image of $F_0$. Since $F_0$ is fully faithful and the latter are isomorphisms, so is the former and we are done.
\end{proof}
\printbibliography
\end{document}